\documentclass[a4paper,12pt]{article}

\usepackage{color,makeidx,bm,latexsym,amscd,amsmath,amssymb,enumerate,stmaryrd,amsthm,float,graphicx,psfrag,geometry}

\usepackage[all]{xy}

\usepackage{tikz,epic,eso-pic}

\theoremstyle{plain}
\newtheorem{theorem}{Theorem}[section]
\newtheorem{corollary}[theorem]{Corollary}
\newtheorem{lemma}[theorem]{Lemma}

\theoremstyle{definition}
\newtheorem{definition}[theorem]{Definition}
\newtheorem{remark}[theorem]{Remark}
\newtheorem{example}[theorem]{Example}

\newcommand{\Q}{\mathbb{Q}}
\newcommand{\R}{\mathbb{R}}

\newcommand{\Hom}{\operatorname{Hom}} 

\newcommand{\ord}{\operatorname{\mathcal{O}}} 

\title{What is $-Q$ for a poset $Q$?} 
\author{Taiga Yoshida\thanks{Department of Mathematics, Graduate School of Science, 
Hokkaido University, North 10, West 8, Kita-ku, 
Sapporo 060-0810, JAPAN 
E-mail: yda.tga@gmail.com}, 
Masahiko Yoshinaga\thanks{Department of Mathematics, Faculty of Science, 
Hokkaido University, North 10, West 8, Kita-ku, 
Sapporo 060-0810, JAPAN 
E-mail: yoshinaga@math.sci.hokudai.ac.jp}}
\date{\today}
\begin{document}
\maketitle

\begin{abstract} 
In the context of combinatorial reciprocity, 
it is a natural question to ask what ``$-Q$'' is for a poset $Q$. 
In a previous work, the definition 
``$-Q:=Q\times\mathbb{R}$ with lexicographic order'' 
was proposed based on the notion of Euler characteristic of 
semialgebraic sets. In fact, by using this definition, 
Stanley's reciprocity for order polynomials was 
generalized to 
an equality for the Euler characteristics of certain spaces of 
increasing maps between posets. 
The purpose of this paper is to refine this result, that is, 
to show that these spaces are homeomorphic if the topology of $Q$ is 
metrizable. 
\end{abstract}


\section{Introduction: Euler characteristic reciprocity}
\label{sec:intro}

For posets $P$ and $Q$, the set of increasing maps from $P$ to $Q$, 
denoted by $\Hom^{<}(P, Q)$, is defined as 
\begin{equation}
\Hom^{<}(P, Q)=\{\eta:P\longrightarrow Q\mid 
p_1<p_2\Longrightarrow \eta(p_1)<\eta(p_2)\}. 
\end{equation}
The set of weakly increasing maps $\Hom^{\leq}(P, Q)$ is similarly defined. 
For finite posets $P$ and $Q$, the cardinality 
$\left|\Hom^{<(\leq)}(P, Q)\right|$ 
is an important object of study in enumerative combinatorics and 
theory of polytopes (\cite{sta-ec}). 
In particular, the following result by Stanley 
is one of the early results which leads 
recent active research 
on combinatorial reciprocities (\cite{bec-san}).
\begin{theorem}
\label{thm:sta}
\cite{sta-chr, sta-ord} 
Let $P$ be a finite poset and 
$[n]$ denote the totally ordered set $\{1<2<\dots <n\}$. Then, 
\begin{itemize}
\item[(i)] (Order polynomials) there exist polynomials 
$\ord^<(P, t), \ord^\leq(P, t)\in\Q[t]$ that satisfy 
\begin{eqnarray}
\ord^\leq(P, n)&=&\left|\Hom^\leq(P, [n])\right|,\\
\ord^<(P, n)&=&\left|\Hom^<(P, [n])\right|, 
\end{eqnarray}
for $n\geq 1$. 
\item[(ii)] (Reciprocity) 
\begin{equation}
\label{recip}
\ord^<(P, t)=(-1)^{|P|}\cdot\ord^\leq(P, -t). 
\end{equation}
\end{itemize}
\end{theorem}
Let $t=n$ in formula (\ref{recip}). The left-hand side 
makes sense in terms of the cardinality of $\Hom^<(P, [n])$. 
However, the right-hand side, the cardinality of $\Hom^{\leq}(P, [-n])$, 
is meaningless as it is. 
For this reason, it is a natural question to give a definition of 
``$-Q$'' for the poset $Q$ and give meaning to the formula of 
the form 
\begin{equation}
\label{expect}
\text{`` }
\#\Hom^<(P, Q)=(-1)^{|P|}\cdot\#\Hom^\leq(P, -Q). 
\text{ ''}
\end{equation}
The cardinality of a finite set is a non-negative integer, 
however for our purposes we need an extension of ``finite sets'' 
such that it takes whole integers (including negative integers) 
as ``cardinality''. Such a problem has been discussed in \cite{sch-neg}, 
and one natural answer is topological spaces 
(in particular, semialgebraic sets) and their Euler characteristics. 
In fact, number of generalizations of combinatorial results have been 
obtained using the Euler characteristic \cite{eas-hug, str-eul}. 

In \cite{hmy}, the definition ``$-Q:=Q\times\R$ with lexicographic order'' was 
proposed for this purpose. 
Then, based on this definition, the above formula (\ref{expect}) 
can be formulated as an identity for the Euler characteristics. 

In order to state the main result of \cite{hmy}, 
let us recall the Euler characteristic of a semialgebraic set \cite{bpr}. 
Let $X\subset\R^N$ be a semialgebraic set. Then, there exists a 
finite partition $X=\bigsqcup_{\lambda\in\Lambda}X_\lambda$ into semialgebraic 
sets $X_\lambda$ which is semialgebraically homeomorphic to 
the open simplex $\sigma_{d_\lambda}$, where 
$\sigma_d=\{0<x_1<x_2<\dots x_d<1\}\subset\R^d$ is the $d$-dimensional 
open simplex (note that $\sigma_0$ is the point). 
Then the Euler characteristic $e(X)$ of $X$ is defined as 
$e(X):=\sum_{\lambda\in\Lambda}(-1)^{d_\lambda}$. 
Note that 
if $X$ is compact, then $e(X)$ coincides with the usual Euler characteristic. 
More generally, if $X$ is locally compact, then $e(X)$ is coincides with 
the Euler characteristic of the Borel-Moore homology group \cite{bcr}. 

A poset is called a \emph{semialgebraic poset} if 
its ground set is a semialgebraic set and 
order structure is semialgebraically defined. 
Finite posets and the real line $\R$ are semialgebraic posets. 
The Euler characteristic of a semialgebraic poset is a natural 
generalization of the cardinality of a finite poset. 
For example, for a finite poset $P$ , 
$e(P)$ is equal to the cardinality $|P|$. 
Furthermore, due to the multiplicativity of the Euler characteristic 
and $e(\R) = -1$, for a semialgebraic poset $Q$, we have 
\[
e(-Q) = -e(Q). 
\]

\begin{theorem}
\label{thm:hmy}
\cite{hmy} Let $P$ be a finite poset and $Q$ be a semialgebraic poset. 
Then $\Hom^<(P, Q)$ and $\Hom^{\leq}(P, Q)$ are semialgebraic sets. 
Furthermore, 
\begin{itemize}
\item[(i)] the Euler characteristics of these spaces satisfy 
\begin{eqnarray}
\label{eq:euler01}
e(\Hom^<(P, Q))=(-1)^{|P|}\cdot e(\Hom^{\leq}(P, -Q)), \\
\label{eq:euler02}
e(\Hom^<(P, -Q))=(-1)^{|P|}\cdot e(\Hom^{\leq}(P, Q)). 
\end{eqnarray}
\item[(ii)] 
If furthermore $Q$ is totally ordered, then the 
Euler characteristics of these spaces 
can be expressed using ordered polynomials as follows. 
\[
\begin{split}
e(\Hom^<(P, Q))&=\ord^<(P, e(Q)), \\
e(\Hom^\leq(P, Q))&=\ord^\leq(P, e(Q)). 
\end{split}
\]
\end{itemize}
\end{theorem}
Note that Stanley's reciprocity Theorem \ref{thm:sta} 
can be obtained by considering the totally ordered set $Q=[n]$. 
Moreover, 
Theorem \ref{thm:hmy} asserts that the reciprocity 
holds for any finite poset $Q$, 
not necessarily for the poset of the form $Q=[n]$. 

\begin{remark}
Note that, in Theorem \ref{thm:hmy} (i), 
since $-(-Q)\neq Q$, the two formulas 
(\ref{eq:euler01}) and (\ref{eq:euler02}) are not equivalent. 
\end{remark}

This paper is organized as follows. In \S \ref{sec:euler}, 
we discuss the refinement of Theorem \ref{thm:hmy} (i), 
i.e., 
whether the claim of the Theorem follows from the homeomorphism of spaces. 
In \S \ref{sec:main} we formulate the main result. 
In \S \ref{sec:semi} we summarize the properties of upper 
semicontinuous functions needed for the proof, and 
in \S \ref{sec:proof} we give the proof of the main result. 

\section{A refinement of Euler characteristic reciprocity}
\label{sec:euler}

Theorem \ref{thm:hmy} (i) asserts that the Euler characteristics 
of two spaces are equal up to sign factor. 
Let us reformulate these formulas: noting that 
$e(\R^{|P|})=(-1)^{|P|}$, 
the two formulas of Theorem \ref{thm:hmy} (i) 
can be rewritten as: 
\begin{eqnarray}
\label{eq:nonhomeo}
e(\Hom^{\leq}(P, Q\times\R))=e(\Hom^<(P, Q)\times\R^{|P|}), \\
\label{eq:homeo}
e(\Hom^<(P, Q\times\R))=e(\Hom^{\leq}(P, Q)\times \R^{|P|}). 
\end{eqnarray}
It is a natural question to ask whether the equality between 
these Euler characteristics can be refined. 
More precisely, are the spaces in the left-hand sides and 
the right-hand sides homeomorphic?


The main result of this paper is to prove that the second equality 
(\ref{eq:homeo}) 
holds at the level of space, that is, there exists a homeomorphism 
\begin{equation}
\label{eq:homeo02}
\Hom^<(P, Q\times\R)\simeq\Hom^{\leq}(P, Q)\times \R^{|P|}. 
\end{equation}
(See Theorem \ref{thm:main} and Corollary \ref{cor:semialg} 
for the precise statement). 

\begin{remark}
For the first equality (\ref{eq:nonhomeo}), 
the spaces 
$\Hom^{\leq}(P, Q\times\R)$ and $\Hom^<(P, Q)\times\R^{|P|}$ 
are not homeomorphic in general. 
For example, when $P=[2]$ and $Q=[1]$, 
$\Hom^<(P, Q)=\emptyset$, therefore, 
the space in the right-hand side of (\ref{eq:nonhomeo}) 
is empty, while, the left-hand side is non-empty. 

As another example, let us consider the case $P=Q=[2]$. 
Then,  $\Hom^<(P, Q)$ consists of a point and 
$\Hom^<(P, Q)\times\R^{|P|}$ is a connected space. 
On the other hand, 
$\Hom^{\leq}(P, Q\times\R)$ looks like 
Figure \ref{fig:example}, which has three connected components. 
(The figure is drawn using the identification $\R$ with the open interval $(0,1)$.)
\begin{figure}[htbp]
\centering
\begin{tikzpicture}[scale=1.2]


\draw[->] (-0.5,0)--(4.5,0);
\draw[->] (0,-0.5)--(0,4.5);
\draw[very thin] (-0.5,-0.5)--(4.5,4.5);

\filldraw[fill=gray!20!white, draw=black, dashed, very thin] 
(1,3)--(2,3)--(2,4)--(1,4)--cycle;

\filldraw[fill=white, draw=black] (1,3) circle (2pt) ;
\filldraw[fill=white, draw=black] (1,4) circle (2pt) ;
\filldraw[fill=white, draw=black] (2,3) circle (2pt) ;
\filldraw[fill=white, draw=black] (2,4) circle (2pt) ;

\filldraw[fill=gray!20!white, draw=black, dashed, very thin] 
(1,1)--(2,2)--(1,2)--cycle;
\draw[very thick] (1,1)--(2,2);
\filldraw[fill=white, draw=black] (1,1) circle (2pt) ;
\filldraw[fill=white, draw=black] (1,2) circle (2pt) ;
\filldraw[fill=white, draw=black] (2,2) circle (2pt) ;

\filldraw[fill=gray!20!white, draw=black, dashed, very thin] 
(3,3)--(4,4)--(3,4)--cycle;
\draw[very thick] (3,3)--(4,4);
\filldraw[fill=white, draw=black] (3,3) circle (2pt) ;
\filldraw[fill=white, draw=black] (4,4) circle (2pt) ;
\filldraw[fill=white, draw=black] (3,4) circle (2pt) ;

\end{tikzpicture}
\caption{$\Hom^{\leq}([2], [2]\times\R)$.}
\label{fig:example}
\end{figure}
It is a natural problem 
to explore the reasons that lead to the equality (\ref{eq:nonhomeo}) 
of the Euler characteristics even though the spaces are not 
homeomorphic. 
\end{remark}

\section{Metrizable posets and main result}
\label{sec:main}

Let $P$ and $Q$ be posets. 
From the definition of lexicographic order, 
a pair $(\eta, \theta)$ of maps 
$\eta:P\longrightarrow Q$ and $\theta:P\longrightarrow\R$ 
is contained in $\Hom^<(P, Q\times\R)$ if and only if for every 
$p_1, p_2\in P$ with $p_1<p_2$, 
either ``$\eta(p_1)<\eta(p_2)$'' or 
``$\eta(p_1)=\eta(p_2)$ and $\theta(p_1)<\theta(p_2)$'' holds. 
It follows that $\eta\in\Hom^\leq(P, Q)$. 
Thus, we obtain the natural projection 
$\pi:\Hom^<(P, Q\times\R)\longrightarrow\Hom^{\leq}(P, Q)$ 
(also similarly 
$\pi:\Hom^\leq(P, Q\times\R)\longrightarrow\Hom^{\leq}(P, Q)$).  

\begin{definition}
A poset $Q$ is a \emph{metrizable poset} if its ground set is equipped with 
metrizable topology. 
\end{definition}
The main result of this paper is as follows. 
\begin{theorem}
\label{thm:main}
Let $P$ be a finite poset and $Q$ be a metrizable poset. Then there exists 
a homeomorphism $\varphi:\Hom^{<}(P,Q\times \R)\stackrel{\simeq}{\longrightarrow}\Hom^{\leq}(P,Q)\times\R^{|P|}$ which makes the following 
diagram commutative: 
\begin{equation}
\label{diag:main}
  \xymatrix{
    \Hom^{<}(P,Q\times \mathbb{R}) \ar[r]^{\varphi} \ar[d]_{\pi} & \Hom^{\leq}(P,Q)\times\mathbb{R}^{|P|} \ar[d]_{\pi} \\
    \Hom^{\leq}(P,Q) \ar[r]^{id} & \Hom^{\leq}(P,Q). 
  }
\end{equation}
\end{theorem}
Before giving the proof, let us discuss special cases of this result. 
\begin{example}
If $Q=[1]$ is the poset with one element, then the result gives a homomorphism 
$\varphi:\Hom^{<}(P, \R)\stackrel{\simeq}{\longrightarrow}\R^{|P|}$. 
Let $P=\{p_1, \dots, p_n\}$. Then $\Hom^{<}(P, \R)$ is expressed as follows. 
\[
\Hom^{<}(P, \R)=\{(t_1, \dots, t_n)\in\R^n\mid t_i<t_j \mbox{ if $p_i<p_j$ in $P$}\}, 
\]
which is a convex open subset of $\R^n$. Hence it is homeomorphic to $\R^n$. 
\end{example}

\begin{example}
Suppose $P=[2]$ and $Q=\R$. Then we have 
\[
\Hom^{<}(P, Q\times\R)=\{((q_1, t_1), (q_2, t_2))\in(\R\times\R)^2\mid q_1<q_2 
\mbox{, or $q_1=q_2$ and $t_1<t_2$}\}. 
\]
Restricting diagram (\ref{diag:main}) to $(q_1, t_1)=(0,0)$, 
Theorem \ref{thm:main} asserts that 
\[
\{(q_2, t_2)\in\R^2\mid q_2>0\mbox{, or $q_2=0, t_2>0$}\}
\]
is homeomorphic to $[0, \infty)\times\R$ (Figure \ref{fig:homeo}). 
\begin{figure}[htbp]
\centering
\begin{tikzpicture}[scale=1]


\fill[fill=gray!20!white] 
(0,0)--(0,4)--(4,4)--(4,0)--cycle;
\draw[very thick] (0,2)--(0,4);
\draw[very thin, dashed] (0,2)--(0,0);
\filldraw[fill=white, draw=black] (0,2) circle (2pt) ;

\fill[fill=gray!20!white] 
(6,0)--(6,4)--(10,4)--(10,0)--cycle;
\draw[very thick] (6,0)--(6,4);

\end{tikzpicture}
\caption{$\{(q_2, t_2)\in\R^2\mid q_2>0\mbox{, or $q_2=0, t_2>0$}\}$ 
and $[0, \infty)\times\R$.}
\label{fig:homeo}
\end{figure}
\end{example}
This example shows that considerations of the upper 
semicontinuous functions are key to the proof of Theorem \ref{thm:main}. 

\section{Upper semicontinuous functions on metrizable spaces}
\label{sec:semi}

Recall that a function $f:X\longrightarrow \R\cup\{-\infty\}$ 
on a topological space $X$ is said to be \emph{upper semicontinuous} if 
for every $\alpha\in\R$, $f^{-1}([-\infty, \alpha))\subset X$ is open. 

\begin{lemma}
\label{lem:upper}
For a function $f:X\longrightarrow \R\cup\{-\infty\}$, let 
$X_f:=\{(x, t)\in X\times\R\mid t>f(x)\}$. If $X$ is metrizable and 
$f$ is upper semicontinuous, then there exists a homeomorphism 
$\varphi: X_f\longrightarrow X\times\R$ 
that makes the following diagram commutative 
\begin{equation}
  \xymatrix{
    X_f \ar[r]^{\varphi} \ar[d]_{\pi} & X\times\R \ar[d]_{\pi} \\
    X \ar[r]^{id} & X. 
  }
\end{equation}
\end{lemma}

\begin{proof}
It is classically known (\cite[Chapter 9, \S 2]{bourbaki}) that 
there exists a sequence $f_n:X\longrightarrow \R\cup\{-\infty\}, (n\geq 1)$ 
of continuous functions such that 
\begin{itemize}
\item for each $x\in X$, $f_1(x)>f_2(x)>\dots>f_n(x)>\dots$, and 
\item $\lim_{n\to\infty}f_n(x)=f(x)$. 
\end{itemize}
Then, define $\varphi:X_f\longrightarrow X\times \R$ as 
\[
\varphi ( x , t ) = 
\begin{cases}
(x, t - f_0(x)) & (t\geq f_0(x) ), \\
(x,-(i+1) + \dfrac{t - f_{i+1}(x)}{f_i(x) - f_{i+1}(x)} & (f_i(x)\geq t \geq f_{i+1}(x)). 
\end{cases}
\]
This $\varphi$ gives a desired homeomorphism. 
\end{proof}

\section{Proof of the main result}
\label{sec:proof}

We give the proof of Theorem \ref{thm:main} in this section. 
We fix a numbering $P=\{p_1, \dots, p_n\}$ 
in such a way that $1\leq i<j\leq n$ implies $p_j\not\leq p_i$. 
Such a numbering can be obtained, for example, by letting 
$p_1$ be a minimal element of $P$ and 
$p_i$ be a minimal element of $P\smallsetminus\{p_1, \dots, p_{i-1}\}$ 
for $i>1$. 

For $1\leq k\leq n$, let us define the subset 
$X_k\subset\Hom^{\leq}(P, Q)\times\R^n$ as follows. 
\[
\begin{split}
X_k:=
\{
(q_1, \dots, q_n, t_1, \dots, t_n)\in\Hom^{\leq}(P, Q)\times\R^n\mid 
&\mbox{For $1\leq \forall i<\forall j\leq k$}, \\
&\mbox{if $p_i< p_j$ and $q_i=q_j$, then $t_i<t_j$}
\}, 
\end{split}
\]
where $(q_1, \dots, q_n)=(\eta(p_1), \dots, \eta(p_n))$ 
for $\eta\in\Hom^{\leq}(P, Q)$ and 
$t_i\in\R$. 
Note that $X_1=\Hom^{\leq}(P, Q)\times\R^n$ and 
$X_n=\Hom^{<}(P, Q\times\R)$. 

Let $1\leq k\leq n-1$. Define the map 
$\pi_k:X_k\longrightarrow\Hom^{\leq}(P, Q)\times\R^{n-1}$ by 
$\pi_k(q_1, \dots, q_n, t_1, \dots, t_n)\longmapsto
(q_1, \dots, q_n, t_1, \dots, t_k, t_{k+2}, \dots, t_n)$, and 
$Y_k:=\pi_k(X_k)$. It follows from the definition that $X_k=Y_k\times\R$. 
Next, for $1\leq j\leq k\leq n-1$, define the function 
$f_{jk}:Y_k\longmapsto\R\cup\{-\infty\}$ as follows. 
\[
f_{jk}(q_1,\ldots,q_n,t_1,\ldots,t_k,t_{k+2},\ldots,t_n) = 
\begin{cases} 
-\infty &
(p_j \nleq p_{k+1} \mbox{ or } q_j < q_{k+1}), \\
t_j &
(p_j \leq p_{k+1} \mbox{ and } q_j = q_{k+1}). 
\end{cases}
\]
Then $f_{jk}$ is an upper semicontinuous function. 
In fact, when $p_j \nleq p_{k+1}$, $f_{jk}$ is upper semicontinuous 
because it is a constant function. 
When $p_j \nleq p_{k+1}$, we need to verify $f_{jk}^{-1}([-\infty, \alpha))$ 
is open for $\forall\alpha\in\R$. Indeed, we have 
\[
\begin{split}
f_{jk}^{-1}([-\infty, \alpha))=
&
\{(q_1, \dots, q_n, t_1, \dots, t_k, t_{k+1}, \dots, t_n)\in Y_k\mid 
q_j\neq q_{k+1}\}\\
&
\cup
\{(q_1, \dots, q_n, t_1, \dots, t_k, t_{k+1}, \dots, t_n)\in Y_k\mid t_j<\alpha\}. 
\end{split}
\]
The first set is open because $Q$ is Hausdorff. The second set is 
clearly open. 

Now we consider the function $f_k:=\max\{f_{1k}, \dots, f_{kk}\}$ on $Y_k$. 
Since the maximum of finitely many upper semicontinuous functions is 
upper semicontinuous, $f_k$ is upper semicontinuous. 
By Lemma \ref{lem:upper}, there exists a homeomorphism $\varphi_i$ 
that makes the following diagram commutative. 
\[
\xymatrix{
Y_k\times \mathbb{R} \ar[r]^{\varphi_k} \ar[d]_{\pi} & 
{Y_k}_{f_k} \ar[d]_{\pi} \\
Y_k \ar[r]^{id} & Y_k, 
}
\]
where ${Y_k}_{f_k}=\{(y, t_{k+1})\in Y_k\times\R\mid t_{k+1}>f_k(y)\}$. 
Furthermore, by definition, we have ${Y_{k}}_{f_k}=X_{k+1}$. 
Hence there exists a homeomorphism $X_k\simeq X_{k+1}$ which commutes 
with the projection to $\Hom^{\leq}(P, Q)$. In particular, we have 
$\Hom^{\leq}(P, Q)\times\R^n=X_1\simeq\dots\simeq X_n=\Hom^<(P, -Q)$. 
This completes the proof of Thoerem \ref{thm:main}. 

Since a semialgebraic set is metrizable, we have the following. 
\begin{corollary}
\label{cor:semialg}
Let $P$ be a finite poset and $Q$ be a semialgebraic poset. Then 
$\Hom^{\leq}(P, Q)\times\R^{|P|}$ and $\Hom^<(P, -Q)$ are homeomorphism. 
\end{corollary}

\begin{remark}
\label{rem:homeo}
It is known that the Euler characteristics of homeomorphic 
semialgebraic sets coincide (\cite{beke}). Therefore, 
the equality of Euler characteristics 
(\ref{eq:euler02}) in Theorem \ref{thm:hmy} 
can be obtained from Corollary \ref{cor:semialg}. 
However, it is not clear whether 
$\Hom^{\leq}(P, Q)\times\R^{|P|}$ and $\Hom^<(P, -Q)$ are 
semialgebraically homeomorphic or not, because in Lemma \ref{lem:upper}, 
we use maps that are not semialgebraic. 
Note that there exist two semialgebraic sets that are homeomorphic, 
but not semialgebraically homeomorphic (\cite{sy-tri}). 
%
\end{remark}

\medskip

\noindent
{\bf Acknowledgements.} 
Masahiko Yoshinaga 
was partially supported by JSPS KAKENHI 
Grant Numbers JP19K21826, JP18H01115. 


\begin{thebibliography}{99}


\bibitem{bpr}
S. Basu, R. Pollack, M. -F. Roy, 
Algorithms in real algebraic geometry. Second edition. 
Algorithms and Computation in Mathematics, 10. 
\emph{Springer-Verlag, Berlin}, 2006. x+662 pp.



\bibitem{bec-san}
M. Beck, R. Sanyal, 
Combinatorial reciprocity theorems. An invitation to enumerative geometric combinatorics. 
Graduate Studies in Mathematics, 195. American Mathematical Society, 
Providence, RI, 2018. xiv+308 pp.


\bibitem{beke}
T. Beke, 
Topological invariance of the combinatorial Euler characteristic of 
tame spaces. 
\emph{Homology Homotopy Appl.} \textbf{13} (2011), no. 2, 165-174.

\bibitem{bcr}
J. Bochnak, M. Coste, M. -F. Roy, 
Real algebraic geometry. 
Ergebnisse der Mathematik und ihrer Grenzgebiete (3) 
36. 
\emph{Springer-Verlag, Berlin}, 1998. 

\bibitem{bourbaki}
N. Bourbaki, General topology: Chapters 5-10, 
Elements of Mathematics (Berlin). Springer-Verlag, Berlin, 1998



\bibitem{eas-hug}
M. Eastwood, S. Huggett, 
Euler characteristics and chromatic polynomials. 
\emph{European J. Combin.} \textbf{28} (2007), no. 6, 1553-1560. 

\bibitem{hmy}
T. Hasebe, T. Miyatani, M. Yoshinaga, 
Euler characteristic reciprocity for chromatic, flow and order polynomials. 
\emph{Journal of Singularities,} \textbf{16} (2017), 212-227.

\bibitem{sch-neg}
S. Schanuel, 
Negative sets have Euler characteristic and dimension. 
\emph{Category theory (Como, 1990)}, 379-385, 
Lecture Notes in Math., 1488, Springer, Berlin, 1991. 

\bibitem{sy-tri}
M. Shiota, M. Yokoi, 
Triangulations of subanalytic sets and locally subanalytic manifolds.
Trans. Amer. Math. Soc. 286 (1984), no. 2, 727-750.

\bibitem{sta-chr}
R. P. Stanley, 
A chromatic-like polynomial for ordered sets. 
1970 Proc. Second Chapel Hill Conf. on Combinatorial Mathematics and its Applications (Univ. North Carolina, Chapel Hill, N.C., 1970) pp. 421-427 Univ. North Carolina, Chapel Hill, N.C. 

\bibitem{sta-ord}
R. P. Stanley, 
Ordered structures and partitions. 
Memoirs of the American Mathematical Society, No. 119. American Mathematical Society, Providence, R.I., 1972. iii+104 pp. 

\bibitem{sta-ec}
R. P. Stanley, 
Enumerative combinatorics. Volume 1 (2nd ed.). 
2012, New York: Cambridge University Press. 



\bibitem{str-eul}
A. W. Strzebonski, 
Euler characteristic in semialgebraic and other o-minimal groups. 
\emph{J. Pure Appl. Algebra} \textbf{96} (1994), no. 2, 173-201.





\end{thebibliography}
\end{document}